\documentclass[a4paper]{amsart}

\usepackage{amssymb,amsmath,graphicx,epstopdf}
\usepackage{enumitem}
\usepackage{framed}
\usepackage[all]{xy}
\usepackage{tikz}
\usepackage{algorithm}
\usepackage{algpseudocode}

\usetikzlibrary{positioning}

\theoremstyle{plain}
\newtheorem{theorem}{Theorem}[section]
\newtheorem{definition}[theorem]{Definition}
\newtheorem{lemma}[theorem]{Lemma}

\newtheorem{proposition}[theorem]{Proposition}

\newtheorem*{openproblem}{Problem}

\newcommand{\CSP}{\operatorname{CSP}}

\begin{document}

\title[New algorithm for few subpowers templates]{A new algorithm for constraint satisfaction problems with few subpowers templates}

\author{Dejan Deli\'c}
\address{Department of Mathematics, Ryerson University,  Canada}
\email{ddelic@ryerson.ca}

\author{Amir El-Aooiti}
\address{Department of Mathematics, Ryerson University, Canada}
\email{amir.elaooiti@ryerson.ca}

\thanks{The first author gratefully acknowledges support by the Natural Sciences and
Engineering Research Council of Canada in the form of a Discovery Grant.}

\begin{abstract} 
In this article, we provide a new algorithm for solving constraint satisfaction problems over templates with few subpowers, by reducing the problem to the combination of solvability of a polynomial number of systems of linear equations over finite fields and reductions via absorbing subuniverses.
\end{abstract}

\maketitle

\section*{Introduction}
\noindent A variety of problems in the fields of combinatorics, artificial intelligence,and programming can be stated within the framework of the Constraint Satisfaction Problems (or, CSPs, for short.)  Presented with an instance of the CSP, the question is to determine whether there is an assignment of values from the prescribed domains to the variables of the instance in such a way, that all the constraints are satisfied. In its full generality, the problem is in the complexity class \textsc{NP}. However, by restricting the attention to a particular subclass of problems whose domains (templates) allow for a particular algebraic characterization, the CSP restricted to the subclass in question may be tractable (i.e. in \textsc{P}.) In this paper, we adopt the following point of view of studying the complexity of CSPs: we study the restrictions of  the instances by allowing a fixed set of constraint relations. This approach is generally referred to in the literature as a \emph{constraint language} or, a \emph{fixed template}  (\cite{b-j-k}). This point of view allows us the access to a variety of algebraic results, which have proved to be indispensable tools in the study of the CSP in the past 15 years or so. 

The algebraic approach has subsequently yielded a number of important results. Among the most important ones, Barto and
Kozik~\cite{b-k2} gave a complete algebraic description of the
constraint languages over finite domains that are solvable by local
consistency methods (these problems are said to be of \emph{bounded
  width}) and as a consequence it is decidable to determine whether a
constraint language can be solved by such methods. 

The motivation underlying the algebraic approach to the study of parametrized CSPs can be stated in the following way: given a constraint language $\Gamma$, what classes of operations preserving the relations in $\Gamma$ guarantee the existence of a ``nice" algorithm solving the problem? 

Another important result in this direction was the one of  A. Bulatov and V. Dalmau (\cite{bulatov2006simple}) proving the existence of such an algorithm for a fairly general and natural class of parametrized CSPs, those whose parametrizing algebra is \emph{Maltsev}, i.e. in which the constraint language is invariant under an algebraic operation satisfying the condition $m(x,x,y)\approx m(y,x,x)\approx y$, for all elements $x$ and $y$ of the algebra. Their algorithm, also known as the Generalized Gaussian Elimination, provided a common generalization for already known algorithms for solving CSPs over affine domains, CSPs on finite groups with near subgroups, etc. The result of Bulatov and Dalmau is based on the algebraic fact that, given any Maltsev algebra $\mathbb{A}$, any subpower of $\mathbb{A}^n$ has a generating set of polynomial (in fact, linear) size in $n$. This approach was further generalized in \cite{BIMMVW} to show that a modification of Bulatov-Dalmau algorithm solves all CSPs over the so called domains \emph{with few subpowers}, i.e. for all the domains $\mathbb{A}$ with the property that any subpower of $\mathbb{A}^n$ has a generating set of polynomial size in $n$. Subsequently, this result was generalized by J. Berman et al. \cite{BIMMVW}, who showed that the Bulatov-Dalmau algorithm can be generalized to a wider class of templates, namely the templates with a $k$-edge polymorphism, which will be defined later. In fact, the class of templates with $k$-edge polymorphisms is precisely the class of templates with few subpowers, i.e. the class of templates to which the mechanism of the Bulatov-Dalmau algorithm can be generalized.

The algorithms presented in \cite{bulatov2006simple} and \cite{BIMMVW} require explicit knowledge of the algebraic operations witnessing the few subpowers property. This, in itself, may be viewed as problematic, since, in practical applications, the CSP is generally presented in the form of its constraint language and computing the required term is a highly nontrivial problem in terms of its complexity. Secondly, those algorithms do not provide ``short" proofs of unsatisfiability in the same way local consistency checks do. Finally, the Generalized Gaussian Elimination algorithm does not make use of the structure theory ofcongruence modular varieties. The algorithm we present in this article is motivated by the desire to address some of these issues.

\section{Preliminaries}

\subsection{Constraint Satisfaction Problem}

\begin{definition} An \emph{instance} of the CSP is a triple $\mathcal{I}=(V,A,\mathcal{C})$, where $V=\{x_1,\ldots,x_n\}$ is a finite set of \emph{variables}, $A$ is a finite domain for the variables in $V$, and $\mathcal{C}$ is a finite set of \emph{constraints} of the form $C=(S,R_S)$, where $S$, the \emph{scope} of the constraint, is a $k$-tuple of variables $(x_{i_1},\ldots,x_{i_k})\in V^k$ and $R_S$ is a $k$-ary relation $R_S\subseteq A^k$, called the \emph{constraint relation} of $C$.

A \emph{solution} for the instance $\mathcal{I}$ is any assignment $f:V\rightarrow A$, such that, for every constraint $C=(S,R_S)$ in $\mathcal{C}$, $f(S)\in R_S$.
\end{definition}

A relational structure $\mathbf{A}=(A, \Gamma)$, defined over the domain $A$ of the instance $\mathcal{I}$, where $\Gamma$ is a finite set of relations on $A$, is often referred to as a \emph{constraint language}, and the relations from $\Gamma$ form the signature of $\mathbf{A}$. An instance of $\CSP (\mathbf{A})$ will be an instance of the CSP such that all constraint relations belong to $\mathbf{A}$. 



\subsection{Basic Algebraic Concepts}\label{sec:Algebra}
\noindent In this subsection, we introduce concepts from universal algebra which will be used in the remainder of the paper. Two good refrences for a more in-depth overview of universal algebra are \cite{Burris1981} and \cite{bergman}.

An \emph{algebra} is an ordered pair $\mathbb{A}=(A, F)$, where $A$ is a nonempty set, the \emph{universe} of $\mathbb{A}$, while $F$ is the set of \emph{basic operations} of $\mathbb{A}$, consisting of functions of arbitrary, but finite, arities on $A$. The list of function symbols and their arities is the \emph{signature} of $\mathbb{A}$. 

A \emph{subuniverse} of the algebra $\mathbb{A}$ is a nonempty subset $B\subseteq A$ closed under all operations of $\mathbb{A}$. If $B$ is a subuniverse of $\mathbb{A}$, by restricting all operations of $\mathbb{A}$ to $B$, such a subuniverse is a \emph{subalgebra} of $\mathbb{A}$, which we denote $\mathbb{B}\leq \mathbb{A}$.

If $\mathbb{A}_i$ is an indexed family of algebras of the same signature, the product $\prod_i \mathbb{A}_i$ of the family is the algebra whose universe is the Cartesian products of their universes $\prod_i A_i$ endowed with the basic operations which are coordinatewise products of the corresponding operations in $\mathbb{A}_i$. If $\mathbb{A}$ is an algebra, its $n$-th Cartesian power will be denoted $\mathbb{A}^n$.

An equivalence relation $\alpha$ on the universe $A$ of an algebra $\mathbb{A}$ is a \emph{congruence} of $\mathbb{A}$, if $\alpha \leq \mathbb{A}^2$, i.e. if $\alpha$ is preserved by all basic operations of $\mathbb{A}$. In that case, one can define the algebra $\mathbb{A}/\alpha$, the \emph{quotient of} $\mathbb{A}$ \emph{by} $ \alpha$, with the universe consisting of all equivalence classes (cosets) in $A/\alpha$ and whose basic operations are induced by the basic operations of $\mathbb{A}$. The $\alpha$-congruence class containing $a\in A$ will be denoted $a/\alpha$.

An algebra $\mathbb{A}$ is said to be \emph{simple} if its only congruences are the trivial, diagonal relation $0_\mathbb{A}=\{(a,a)\, \vert \, a\in A\}$ and the full relation $1_\mathbb{A}=\{ (a,b)\, \vert\, a,b\in A\}$. 

Any subalgebra of a Cartesian product of algebras $\mathbb{A}\leq \prod_i \mathbb{A}_{i\in I}$ is equipped with a family of congruences arising from projections on the product coordinates. We denote $\pi_i$ the congruence obtained by identifying the tuples in $A$ which have the same value in the $i$-th coordinate. Given any $J\subseteq I$, we can define a subalgebra of $\mathbb{A}$, $proj_J(\mathbb{A})$, which consists of the projections of all tuples in $A$ to the coordinates from $J$. If $\mathbb{A}\leq \prod_{i\in I} \mathbb{A}_i$ is such that $proj_i (\mathbb{A})=\mathbb{A}_i$, for every $i\in I$, we say that $\mathbb{A}$ is a \emph{subdirect product} and denote this fact $\mathbb{A}\leq_{sp} \prod_{i\in I} \mathbb{A}_i$.

If $\mathbb{A}$ and $\mathbb{B}$ are two algebras of the same signature, a mapping from $A$ to $B$ which preserves all basic operations is a \emph{homomorphism}. An \emph{isomorphism} is a bijective homomorphism between two algebras of the same signature.

Given an algebra $\mathbb{A}$, a \emph{term} is a syntactical object describing a composition of basic operations of $\mathbb{A}$. A \emph{term operation} $t^\mathbb{A}$ of $\mathbb{A}$ is the interpretation of the syntactical term $t(x_1,\ldots,x_m)$ as an $m$-ary operation on $A$, according to the formation tree of $t$.

A \emph{variety} is a class of algebras of the same signature, which is closed under the class operators of taking products, subalgebras, and homomorphic images (or, equivalently, under the formation of quotients by congruence relations.) We denote these class operators $\operatorname{H}$, $\operatorname{S}$, and $\operatorname{P}$, respectively. The variety $\mathcal{V}(\mathbb{A})$ generated by the algebra $\mathbb{A}$ is the smallest variety containing $\mathbb{A}$ and, 
$$\mathcal{V}(\mathbb{A})=\operatorname{H}\operatorname{S}\operatorname{P}(\mathbb{A}).$$

 Birkhoff's theorem (see \cite{Burris1981}) states that every variety is an equational class; that is, every variety $\mathcal{V}$ is uniquely determined by a set of identities (equalities of terms) $s\approx t$ so that $\mathbb{A}\in\mathcal{V}$ if and only if $\mathbb{A}\models s\approx t$, for every identity $s\approx t$ in the set.

An \emph{$n$-ary operation} on a set $A$ is a mapping
$f:A^n\rightarrow A$; the number $n$ is the \emph{arity} of $f$.  Let
$f$ be an $n$-ary operation on $A$ and let $k>0$. We write $f^{(k)}$
to denote the $n$-ary operation obtained by applying $f$ coordinatewise on
$A^k$. That is, we define the $n$-ary operation $f^{(k)}$ on $A^k$ by
\[
f^{(k)}(\mathbf a^1,\dots,\mathbf
a^n)=(f(a^1_1,\dots,a^n_1),\dots,f(a^1_k,\dots,a^n_k)),
\]
for $\mathbf a^1,\dots, \mathbf a^n\in A^k$.

The notion of \emph{polymorphism} plays the central role in the 
algebraic approach to the $\CSP$. 

\begin{definition}
  Given an $\Gamma$-structure $\mathbf{A}$, an $n$-ary
  \emph{polymorphism} of $\mathbf{A}$ is an $n$-ary operation $f$ on
  $A$ such that $f$ preserves the relations of $\mathbf A$. That is,
  if $\mathbf{a}^1,\dots,\mathbf{a}^n\in R$, for some $k$-ary relation
  $R$ in $\Gamma$, then $f^{(k)}(\mathbf a^1,\dots,\mathbf
  a^n)\in R$.  
\end{definition}

If a relational structure $\mathbf{A}$ is a core, one can construct a structure $\mathbf{A}'$ from $\mathbf{A}$ by adding, for each element $a\in A$, a unary constraint relation $\{a\}$. This enables us to further restrict the algebra of polymorphisms associated with the template; namely, if $f(x_1,\ldots,x_m)$ is an $m$-ary polymorphism of $\mathbf{A}'$, it is easy to see that
$f(a,a,\ldots,a)=a,$
for all $a\in A$. In addition to this, the constraint satisfaction problems with the templates $\mathbf{A}$ and $\mathbf{A}'$ are logspace equivalent. Therefore, we may assume that the algebra of polymorphisms associated to any CSP under consideration is \emph{idempotent}; i.e. all its basic operations $f$ satisfy the identity
$$f(x,x,\ldots,x)\approx x.$$

\subsection{Absorption}\label{absorption}

One of the key notions which has emerged in recent years as an important tool in the algebraic approach to the study of CSPs with finite templates is the one of absorption. It has played a crucial role in the proof of the Bounded Width Conjecture and its refinements (see \cite{b-k1} , \cite{b-k2},  \cite{Kozik2016}) but its primary strength is in its applicability outside the context of congruence meet-semidistributivity.

If $\mathbb{A}$ and $\mathbb{B}$ are idempotent algebras such that $\mathbb{B}\leq \mathbb{A}$, we say that $\mathbb{B}$ \emph{absorbs} $\mathbb{A}$ and write it as $\mathbb{B}\unlhd \mathbb{A}$ if there exists a term $t$
such that
$$t(B,B,\ldots, B,A,B,\ldots, B)\subseteq B,$$
regardless of the placement of $A$ in the list of variables of the term.

A direct consequence of the definition is the following fact: if $\mathbb{A},\mathbb{B},\mathbb{A}'$ and $\mathbb{B}'$ are algebras of the same signature such that $\mathbb{B}\unlhd\mathbb{A}$ and $\mathbb{B}'\unlhd\mathbb{A}'$, then both absorptions can be witnessed by the same term.

\begin{proposition}\label{absorb} Let $\mathbb{R}\leq_{sp} \mathbb{A}\times \mathbb{B}$.
\begin{enumerate}
\item The binary relation $\alpha$ defined on $A$ by
$$(a,a')\in\alpha  \mbox{ if and only if  there exists $b\in B$ such that } (a,b),(a,b')\in C$$
is a congruence of $\mathbb{A}$. The analogous statement is true of the dual relation $\beta$ defined on $B$.
\item If $\mathbb{C}'\unlhd \mathbb{C}\leq_{sp} \mathbb{A}\times\mathbb{B}$ and $\mathbb{C}'\leq_{sp}\mathbb{A}\times\mathbb{B}$, if $\alpha'$ and $\beta'$ are the pair of congruences defined on $A'$ and $B'$, respectively, as in (1), then $\alpha=\alpha'$ and $\beta=\beta'$.
\end{enumerate}
\end{proposition}

We will refer to the congruences $\alpha$ and $\beta$, defined as in Part (1) of the Proposition \ref{absorb} , as the \emph{linkedness} congruences on $\mathbb{A}$ and $\mathbb{B}$ induced by $\mathbb{C}$. We say that $\mathbb{A}$ and $\mathbb{B}$ are \emph{linked} if $\alpha=1_\mathbb{A}$ and $\beta=1_\mathbb{B}$ or, equivalently, if $\pi_1\vee\pi_2=1_\mathbb{C}$. If $\alpha=0_\mathbb{A}$ and $\beta=0_\mathbb{B}$, the subdirect product is the graph of an isomorphism between the algebras $\mathbb{A}$ and $\mathbb{B}$.

For Taylor algebras, linked subdirect products satisfy the following property:

\begin{theorem} (L. Barto, M. Kozik, \cite{b-k2}) \label{AbsThm} Let $\mathbb{C}\leq_{sp} \mathbb{A}\times\mathbb{B}$ be a Taylor algebra. If $\mathbb{C}$ is linked then
\begin{itemize}
\item $\mathbb{C}=\mathbb{A}\times\mathbb{B}$, or
\item $\mathbb{A}$ has a proper absorbing subalgebra, or
\item $\mathbb{B}$ has a proper absorbing subalgebra.
\end{itemize}
\end{theorem}

\section{Datalog, Linear Arc Consistency, and Singleton Linear Arc Consistency}

A \emph{Datalog program} for a relational template $\mathbf{A}$ is a finite set of rules of the form
$$T_0\leftarrow T_1,T_2,\ldots, T_n$$
where $T_i$'s are atomic formulas. 
$T_0$  is the \emph{head} of the rule, while $T_1,T_2,\ldots, T_n$ form the \emph{body} of the rule.
Each Datalog program consists of two kinds of relational predicates:
the \emph{intentional} ones (IDBs), which are those occurring at least once in the head of some rule and which are not part of the original signature of the template (they are derived by the computation.)
The remaining predicates are said to be the \emph{extensional} ones, or EDBs. They are relations from the signature of the template and do not change during computation; i.e. they cannot appear in the head of any rule.
In addition to those, there is one special, designated IDB, which is nullary (Boolean) and referred to as the \emph{goal} of the program.

We say that the rule
$$T_0\leftarrow T_1,T_2,\ldots, T_n$$
is \emph{linear} if at most one atomic formula in its body is an IDB. A Datalog program is linear if so are all its rules. 

The semantics of Datalog programs are generally defined
in terms of fixed-point operators. We are particularly interested in the Datalog programs which, being presented a relational template $\mathbf{A}$, verify if the template satisfies certain consistency requirements in terms of witnessing path patterns prescribed by the CSP instance in question. 

\subsection{Linear Arc Consistency}

Given a CSP instance $\mathcal{I}$ over a relational template $\mathbf{A}$, a Datalog program verifying its linear arc consistency has one IDB $B(x)$, for each subset $B\subseteq A$ in the instance. To construct rules for the program, we consider a single constraint $R(x_{i_1},x_{i_2},\ldots, x_{i_m})$, with $R$ being a $k$-ary relation in the signature of $\mathbf{A}$, and two variables $x_{i_j},x_{i_k}$ in its scope. If a fact $B(x_{i_j})$ has already been established about $x_{i_j}$, we add the rule
$$C(x_{i_k}) \leftarrow R(x_{i_1},x_{i_2},\ldots, x_{i_m}), B(x_{i_j}).$$
The collection of all such rules, along with the goal, is said to be a Datalog program verifying the \emph{linear arc consistency} of the instance. If the goal predicate is derived, the instance is not linearly arc consistent; otherwise, we say that it  has linear arc consistency, or LAC, for short.

The complexity of verifying LAC for an instance is in nondeterministic log-space, since it reduces to verifying reachibility in a directed graph.

\subsection{Singleton Linear Arc Consistency}\label{SLAC}

Singleton linear arc consistency (or, SLAC, for short) is a consistency notion provably stronger than linear arc consistency. A recent result of M. Kozik (\cite{Kozik2016}) proves that, in fact, all CSPs over the templates of bounded width can be solved by SLAC, whereas, under the assumption that \textsc{NL}$\neq$ \textsc{P}, there are CSPs over the bounded width templates which cannot be solved by LAC, for instance \textsc{3-HORN-SAT}, the satisfiability of Horn formulas in the 3-CNF.

We describe the algorithm for verifying SLAC in its procedural form. Given an instance $\mathcal{I}$, we introduce a unary constraint $B_x$, for each variable in the instance and update them by running the LAC algorithm with the value of $x$ being fixed to an arbitrary $a\in B_x$.

\begin{algorithm}[H]
   \caption{SLAC Algorithm}
    \begin{algorithmic}[1]
     
        \For{ every variable $x$ of $\mathcal{I}$}
            \State Introduce the unary constraint $B_x :=(x,A)$
        \EndFor
        \Repeat
                \For{ every variable $x$ of $\mathcal{I}$}  
                    \State $C:=A$
                    \For{ every value $a\in A$}
                          \State run LAC on the restriction of $\mathcal{I}$ with $B_x=\{a\}$ and constraints modified accordingly
                          \If { LAC results in contradiction}
                                \State remove $a$ from $C$
                         \EndIf
                    \EndFor
                    \State $B_x:=(x,C)$
                 \EndFor
        \Until{ There are no further changes in $B_x$}
\end{algorithmic}
\end{algorithm}

In this paper, we will be using the multisorted version of SLAC. What we mean by that, is that the predicates for the domains of different variables $x$ are assumed to be the subsets of different sorted domains, generated by the reduction to a binary instance. Since the domains produced by the reduction to the binary case are positive-primitive definable, this presents no particular issue.

\section{Patterns and steps}

We will create SLAC instances of structures with binary constraints, and, to that end, we define the notions of a pattern and a step. Our definitions will be special cases of the more general ones given in \cite{Kozik2016}. We fix an instance $\mathcal{I}$ of a CSP, all of whose constraint relations are binary.

\begin{definition} A \emph{step}  in an instance $\mathcal{I}$ is a pair of variables which is the scope of a constraint in $\mathcal{I}$. A \emph{path-pattern} from $x$ to $y$ in $\mathcal{I}$ is a sequence of steps such that every two steps correspond to distinct binary constraints and which identifies each step's end variable with the next step's start variable. A \emph{subpattern} of a path-pattern is a path-pattern defined by a substring of the sequence of steps. We say that a path-pattern is a \emph{cycle} based  at $x$ if both its start and end variable are $x$.
\end{definition}

\begin{definition} Let
$$p=(x_1,x_2,\ldots,x_k)$$
be a path-pattern. A \emph{realization} of $p$ is a $k$-tuple $(a_1,\ldots,a_k)\in \mathbb{S}_{x_1}\times\ldots\times \mathbb{S}_{x_k}$ such that $(a_i,a_j)$ satisfies the binary constraint associated with the $(x_i,x_j)$-step. If $p$ is a path-pattern with the start variable $x_i$ and $A\subseteq S_{x_i}$, we denote $A+p$ the set of the end elements of all realizations of $p$ whose first element is in $A$. $-p$ will denote the inverse pattern of $p$, i.e. the pattern obtained by reversing the traversal of the pattern $p$. In that case, we define $A-p = A+(-p)$.
\end{definition}

The following observations follow directly from the definitions of the notions of Linear Arc Consistency and Singleton Linear Arc Consistency:

\begin{enumerate} 
\item The LAC algorithm does not derive a contradiction on the instance $\mathcal{I}$ if and only if every path-pattern in $\mathcal{I}$ has a solution.
\item If an instance $\mathcal{I}$ is a SLAC instance then, for every variable $x$ and every $a\in \mathbb{S}_x$, and every path pattern $p$ which is a cycle based at $x$, there exists a realization of $p$ with $x$ being assigned the value  $a$. 
\end{enumerate}

\section{Algebras with few subpowers}\label{skewfree}

In this section, we  give a brief overview of the main properties of algebras with the so-called $k$-edge terms and the algebraic properties implied by their existence. In particular, we highlight the relationship between the existence of $k$-edge terms and the property of having few subpowers. For more detials, the reader is referred to \cite{BIMMVW}.

\begin{definition}
For $k \geq 2$, a \textit{$k$-edge operation} on a set $A$ is a $(k+1)$-ary operation $e$ which, for all $x,y \in A$, satisfies:
\begin{align*}
e(x,x,y,y,y,...,y,y) &= y \\
e(x,y,x,y,y,...,y,y) &= y \\
e(y,y,y,x,y,...,y,y) &= y \\
e(y,y,y,y,x,...,y,y) &= y \\
&\hspace{0.2cm}\vdots \\
e(y,y,y,y,y,...,x,y) &= y \\
e(y,y,y,y,y,...,y,x) &= y.
\end{align*}
\end{definition}

It is clear that an operation is a 2-edge operation if and only if it is a Maltsev operation. It was also established in \cite{BIMMVW} that if an algebra contains a $k$-ary GMM (generalized majority-minority) operation, then it also contains a $k$-edge operation. As a result, we can see that a $k$-edge operation is indeed a generalization of both Maltsev and GMM operations. 

\begin{definition}
An algebra $\mathbb{B}$ is a \textit{subpower} of $\mathbb{A}$ if $\mathbb{B}$ is a subalgebra of $\mathbb{A}^n$ for some positive integer $n$. 
\end{definition}

If $\mathcal{I}$ is an instance of $\operatorname{CSP}(\mathbb{A})$ in the set of variables $V$, where $n = |V|$,  its solution set is a subpower of $\mathbb{A}$ (i.e. a subuniverse of $\mathbb{A}^n$.)  The Bulatov-Dalmau algorithm requires that the size of a generating set for the solution subuniverse of $\mathbb{A}^n$ be ``small", meaning that it is of order $\mathcal{O}(n^p)$ for some positive integer $p$. We mention certain equivalent conditions which ensure that all subpowers of $\mathbb{A}$ have small generating sets.

\begin{definition} Let $\mathbb{A}$ be an algebra and let $\mathcal{S} = \{\mathbb{B} \mid \mathbb{B} \leq \mathbb{A}^n\}$ be the set of all subpowers of $\mathbb{A}$ for some positive integer $n$. Let $s_\mathbb{A} (n) = \text{log}_2 |\mathcal{S}|$ and let $g_\mathbb{A} (n)$ be the smallest integer $t$ such that every subpower of $\mathbb{A}$ has a generating set of size at most $t$. For some positive integer $p$,
\begin{enumerate}
\item if $s_\mathbb{A} (n) = \mathcal{O}(n^p)$, then $\mathbb{A}$ has \textit{few subpowers}, and
\item if $g_\mathbb{A} (n) = \mathcal{O}(n^p)$, then $\mathbb{A}$ has \textit{polynomially generated subpowers}.
\end{enumerate}
\end{definition}

\begin{theorem}
Let $\mathcal{I}$ be an instance of $\operatorname{CSP}(\mathbb{A})$ and let $n = |V|$. The following statements are all equivalent for the algebra $\mathbb{A}$:
\begin{enumerate}
\item $\mathbb{A}$ contains a $k$-edge operation for some $k \geq 2$.
\item $\mathbb{A}$ has few subpowers and $s_\mathbb{A} (n) = \mathcal{O}(n^k)$.
\item $\mathbb{A}$ has polynomially generated subpowers and $g_\mathbb{A} (n) = \mathcal{O}(n^{k-1})$.
\end{enumerate}
\end{theorem}

\noindent The proof for this theorem utilizes three term operations in $\mathbb{A}$ derived from the $k$-edge operation. The presence of these term operations allows us to derive the followingfact.

\begin{lemma}(\cite{BIMMVW})
Let $\mathbb{A}$ be a finite algebra with universe $A$ which contains a $k$-edge term operation $e$. Then the algebra $(A,e)$ contains three term operations $d(x,y)$, $p(x,y,z)$, and $s(x_1, x_2, ... , x_k)$ (which are also term operations of $\mathbb{A}$) such that:
\begin{align*}
p(x,y,y) &= x \\
p(x,x,y) &= d(x,y) \\
d(x, d(x,y)) &= d(x,y) \\
s(y,x,x,x,x,...,x,x) &= d(x,y) \\
s(x,y,x,x,x,...,x,x) &= x \\
s(x,x,y,x,x,...,x,x) &= x \\
&\hspace{0.2cm}\vdots \\
s(x,x,x,x,x,...,x,y) &= x. \\
\end{align*}
\par
\end{lemma}

In the same article, the authors prove the following (Theorems 7.1 and 8.1):

\begin{theorem} \label{NU} If $\mathbb{A}$ is a finite algebra which generates a variety omitting types 1 and 2, which has a $k$-edge term, for some $k\geq 3$, then the variety $V(\mathbb{A})$ is congruence distributive. In fact, $\mathbb{A}$ has a $k$-ary near unanimity term $s$.
\end{theorem}

In subsequent sections, we assume that all algebras under consideration have a $k$-edge term, for some fixed $k\geq 2$.

\section{Reduction to binary relations}\label{Binary}

In this section, we outline the reduction of an arbitrary instance with a sufficient degree of consistency to a binary one. The construction is due to L. Barto and M. Kozik and we largely adhere to their exposition in \cite{b-k1}.

An instance is said to be \emph{syntactically simple} if it satisfies the following conditions:

\begin{itemize}
\item every constraint is binary and it its scope is a pair of distinct variables $(x,y)$.
\item for every pair of distinct variables $x,y$, there is at most one constraint $R_{x,y}$ with the scope $(x,y)$.
\item if $(x,y)$ is the scope of $R_{x,y}$, then $(y,x)$ is the scope of the constraint $R_{y,x}=\{(b,a) \, \vert\, (a,b)\in R_{x,y}\}$ (\emph{symmetry of constraints}).
\end{itemize}

Given the congruence modular algebra $\mathbb{A}$ parametrizing the instance $\mathcal{I}$, such that the maximal arity of a relation in $\mathcal{I}$ is $p$, and which has a $k$-ary edge term, let $K=\max\{p,k-1\}$. We run the algorithm verifying the $(2\lceil\frac{K}{2}\rceil,3\lceil\frac{K}{2}\rceil)$-consistency on $\mathcal{I}$. If the algorithm terminates in failure, we output ``$\mathcal{I}$ has no solution." If the algorithm terminates successfully, we output a new, syntactically simple instance $\mathcal{I}'$ in the following way:

\begin{itemize} 
\item The instance is parametrized by $\mathbb{A}^{\lceil \frac{p}{2}\rceil}$, which is an algebra with a $k$-edge term, for some $k\geq 3$. Since $\mathbb{A}$ generates a  variety, which has a $k$-edge term and, then, so does the variety generated by $\mathbb{A}^{\lceil\frac{p}{2}\rceil}$.
\item For every $\lceil\frac{p}{2}\rceil$-tuple of variables in $\mathcal{I}$, we introduce a new variable in $\mathcal{I}'$ and, if $x=(x_1,\ldots,x_{\lceil\frac{p}{2}\rceil})$ and $y=(y_1,\ldots,y_{\lceil\frac{p}{2}\rceil})$ with $x\neq y$, we introduce a constraint 
\begin{multline*}
R_{x,y}=\{((a_1,\ldots,a_{\lceil\frac{p}{2}\rceil}),(b_1,\ldots,b_{\lceil\frac{p}{2}\rceil}))\,\vert \\ (a_1,\ldots,a_{\lceil\frac{p}{2}\rceil},b_1,\ldots,b_{\lceil\frac{p}{2}\rceil}) \mbox{ admit a consistent $p$-assignment of values }\}.
\end{multline*}
\end{itemize}

The binary instance $I'$ constructed in this way will have a solution if, and only if, the instance $I$ has a solution. and will be (2,3)-consistent.

\section{A consistency algorithm for CSPs with few subpowers templates}

In this section, we will construct a polynomial time algorithm for solving constraint satisfaction problems over a finite  template with few subpowers.

The underpinning of the algorithm is the notion of affine consistency which pre-processes the instance by removing affine CSP subinstances which do not contain any solutions. The consistency check will produce a polynomial family of subinstances in which solutions must lie. We can view such subinstances as ``passive"; the reductions performed on the current active instance are implicitly performed on the passive subinstances in such a way that affine consistency is preserved. A reader familiar with the Few Subpowers Algorithm will notice a smiliarity between the passive subinstances and the notion of compact representation (i.e. signature) in that algorithm. This similarity is not coincidental; the tuples in the compact representation of a relation will belong to passive subinstances.

\subsection{Affine consistency}

Let $A\leq \mathbb{S}_x$. For any $y\in V$, $y\neq x$, we define $R^+_{x,y}(A)=\{b\in \mathbb{S}_{y} \, : \, \exists a\in A, \, (a,b)\in R_{x,y}\}$.  Clearly, $R^+_{x,y}(A)$ is a subuniverse of $\mathbb{S}_{y}$.

We will define affine consistency in such a way, that all subinstances which are CSPs over isomorphic simple subuniverses and which do not have a solution, are excluded by the consistency check. In addition, such subinstances which cannot be extended to a larger (2,3)-consistent subinstance are excluded.

We start by defining the list $\mathcal{A}$ of all pairs $(A,\theta_A)$, where $A$ is a subuniverse of some $\mathbb{S}_{x_i}$ and $\theta_A$ is a maximal congruence of $A$.We can arrange the list $\mathcal{A}$ in such a way that, if  $(A,\theta_A)$ and $(A',\theta_{A'})$ are two elements of the list and $A'$ is contained in a $\theta_A$-block of $A$, then $(A,\theta_A)$ appears in the list $\mathcal{A}$ before $(A',\theta_{A'})$.

We are now ready to state the procedure which enforces affine consistency

\begin{enumerate}
\item For the next pair $(A,\theta_A)$ in the list $\mathcal{A}$, form the $(A,\theta_A)$-\emph{test instance} in the following way: suppose $A\leq \mathbb{S}_{x}$, for some $x\in V$. For $y\neq x$, if there exists a congruence $\alpha_{x_j}$ on $R_{x,y}(A)$, such that, if $B_1$ and $B_2$ are two distinct $\theta_A$-blocks and $p$ a path pattern from $x_i$ to $x_j$ such that $R^+_{x,y}(A)\cap (B_1+p)$ and $R^+_{x,y}(A)\cap (B_2+p)$ are containt in distinct blocks of $\alpha_{y}$, we will say that the variable $y$ is  \emph{relevant}. Therefore, for each relevant variable $y$,
$$\mathbb{S}_y/\alpha_{y}\cong A/\theta_A.$$
In fact, $\alpha_{y}$ is independent of the choice of the path pattern $p$, because of (2,3)-consistency. 

We define a \emph{strand} to be the set of those congruence blocks in each relevant domain which are linked to the same congruence block of $\theta_A$. The $(A,\theta_A)$-test instance will have as its domains the algebras $R^+_{x,y}(A)/\alpha_{y}$, for $j\neq i$, for all relevant variables $y$ and $A$ in the $x$-th coordinate, along with $\alpha_{x}=\theta_A$.  Since $\mathcal{I}$ is a (2,3)-consistent instance, for any pair of relevant variables $y,z$, distinct from $x$, the binary constraint $E_{y,z}$ induces a subdirect product on $R_{x,y}^+(A)$ and $R_{x,z}^+(A)$, so that the $(A,\theta_A)$-test instance is 1-consistent.

\item The $(A,\theta_A)$-test instance is a  syntactically simple binary CSP over a simple affine module. Based on an earlier discussion, a simple affine module has as its subreduct the structure of a finite dimensional vector space over a finite field. As such, it can be solved using Gaussian elimination since it is, essentially, a system of linear equations over the said finite field. Given a $\theta_A$-block $B$, we test whether $B$ appears in any solution of the system. If it does not, we remove all vertices in $B$ from the instance.

\item For each $(A,\theta_A)$-block $B$ which passed the test in the previous step, form a subinstance whose domains are $B$ and $R^+_{x,y}(B)$, for $y\neq x$ and enforce (2,3)-consistency on it. If the (2,3)-consistency test fails, remove the block $B$ from the instance.

\item Enforce (2,3)-consistency.

\item Proceed to the next element in the list $(A',\theta_A')$, if there are any left.

\end{enumerate}

There are only polynomially many pairs in the list $\mathcal{M}$, so the algorithm for enforcing affine consistency runs in polynomial time. In fact the number of test instances can be bounded above by $\mathcal{O}(n)$.

\begin{lemma} Let $\mathcal{I}$ be a syntactically simple binary instance and let $\mathcal{I'}$ be the instance produced by applying the affine consistency check to it. Then, the sets of solutions to $\mathcal{I}$ and $\mathcal{I'}$ coincide.
\end{lemma}

\begin{proof} If there exists a solution $f$ to $\mathcal{I}$ whose projection to the $x$-coordinate is in $A\leq \mathbb{S}_x$, then, its restriction to relevant variables is also a solution of the $(A,\alpha_A)$-test instance, viewed as a subinstance of $\mathcal{I}$. If affine consistency test fails on an $\alpha_A$-block, then there cannot be any solutions $f$ projecting into that block in their $x$-coordinate.

Also, the solution projecting into a $\alpha_A$-block $B$ in its x-coordinate will lie in its entirety in the subinstance induced by $B$, so this subinstance must be (2,3)-consistent.
\end{proof}

Assuming the instance passes the affine consistency test, we will also obtain a polynomial family of subinstances $\mathcal{P}$, henceforth referred to as \emph{passive subinstances}, which are those subinstances examined in Step 3 of the algorithm, and which are 1-consistent, after the affine consistency check has terminated successfully.

\section{Reduction to smaller subinstances - outline}

In what follows, we assume that $\mathcal{I}$ is a 1-consistent SLAC instance which is also affine consistent, with the accompanying list $\mathcal{P}$ of passive subinstances, such that, for every subuniverse $A\leq \mathbb{S}_{x}$, $x\in V$, and every maximal congruence $\theta$ of $A$, such that $A/\theta_A$ is a simple affine module, there is a passive 1-consistent SLAC subinstance generated by every $\theta$-block of $A$.

\section{Instance $\mathcal{I}$ is not absorption-free}

In this case,  some $\mathbb{S}_{x_i}$ contains a proper absorbing subuniverse $B$. The reduction via absorption from Kozik's paper adapts to this setting and can be applied to the instance $\mathcal{I}$. This particular choice of $B$ implicitly defines reductions on all passive subinstances from $\mathcal{P}$. If the reduced subinstance $\mathcal{I}'$ fails to intersect a passive subinstance from $\mathcal{P}$, that passive subinstance is removed from $\mathcal{P}$. For the instance $\mathcal{I}$, if $\mathbb{S}_x$ is not absorption-free, the analysis of the proof in Section 10 of \cite{Kozik2016} indicates that $B\unlhd \mathbb{S}_x$ can always be chosen in such a way that $B$ is a minimal absorbing subuniverse of $\mathbb{S}_x$ and which is, therefore, absorption-free.

It is easily seen that, for any such choice of $B$, $\mathcal{P}$ cannot become empty: namely, by considering any maximal congruence $\psi$ on $B$, we see that the passive subinstances determined by the blocks of $B/\psi$ must remain in $\mathcal{P}$.

The following fact has an obvious proof, based on the definition of an absorbing subuniverse:

\begin{lemma} Let $B\leq A$ and $C\unlhd A$. Then, if $C\cap B\neq\emptyset$, $C\cap B\unlhd B$.
\end{lemma}

\begin{proposition} \label{useful} Let $\mathcal{I}$ be a 1-consistent syntactically simple binary instance with domains $\mathbb{S}_x$, $x\in V$. If $A\unlhd \mathbb{S}_x$, then $R^+_{x,y}(A)\unlhd \mathbb{S}_y$, for all $y\neq x$.
\end{proposition}

Using this fact, we see that, if $M\leq \mathbb{S}_x$ is absorption-free, then every absorbing subuniverse of $C\unlhd \mathbb{S}_x$, such that $C\cap M\neq\emptyset$, must satisfy $M\leq C$. In addition, for every $y$, $R^+_{x,y}(M)\leq R^+_{x,y}(C)$. The proof from \cite{Kozik2016} shows that SLAC remains preserved in all passive subinstances under the absorption reduction defined in that paper, unless a domain of the passive subinstance fails to intersect the minimal absorbing subuniverse in $\mathbb{S}_{x}$ in some coordinate $x$.

\section{Absorption-free Instances}

In this section, we assume that the instance $\mathcal{I}$ is such, that all domains $\mathbb{S}_x$, $x\in V$ are absorption-free.

\subsection{There exists a simple affine module in the instance}

Let $\theta_x$ be a maximal congruence of $\mathbb{S}_x$, such that $\mathbb{S}_x/\theta_x$ contains a simple affine module $A$ as a subuniverse. Since affine consistency has been enforced on $\mathcal{I}$, for any $\theta_x$-block $B$ contained in $A$, we can select a passive subinstance from $\mathcal{P}$, which is $B$-based and replace $\mathcal{I}$ with it, to obtain a proper subinstance $\mathcal{I}'$, which is SLAC.

In addition, we update all the passive subinstances in $\mathcal{P}$ by taking the intersections with the $B$-based subinstance.

One thing that remains to be verified is that, after the reduction of $\mathcal{I}$ to the $B$-based subinstance $\mathcal{I}'$, if $y$ was a non-relevant variable in some subinstance from $\mathcal{P}$, it remains such after the reduction. 

Let $A'$ be a subuniverse of some $\mathbb{S}_x$ and $\theta$ the congruence such that $y$ is a non-relevant variable in the $(A',\theta)$-test instance. Then, $A'/\theta$ and $R^+_{x,y}(A')$ are linked and the subproduct of $A'/\theta$ with every minimal absorbing subuniverse of $R^+_{x,y}(A')$ will also be linked and, therefore, a full product.

\subsection{There is no simple affine module in the quotient instance}

Next, we refer to the analysis of simple, idempotent, congruence skew-free algebras with an edge term from Section \ref{skewfree}. 

In this case, $HS(\mathbb{S}_x)$ contains no simple affine modules, for all $x\in V$. Then, by \cite{lvz}, the variety $\mathcal{V}(\mathbb{S}_x)$ omits types 1 and 2 in the sense of tame congruence theory. Hence, $\mathbb{S}_x$ is of bounded width and, in fact, the pseudovariety $\operatorname{HSP}_{fin}(\{\mathbb{S}_x \, :\, x\in V\})$ also omits types 1 and 2. by Theorem \ref{NU}, the pseudovariety has an $l$-ary NU polymorphism, for some $l\geq 3$.

For that reason, since the instance $\mathcal{I}$ is 1-consistent and SLAC, by \cite{Kozik2016}, the instance contains a solution. In fact, since the instance $\mathcal{I}$ contains an NU polymorphism, a weaker condition than SLAC, linear arc consistency, is sufficient to guarantee the existence of a solution.

\section{Conclusion}

We have presented an algorithm for solving constraint satisfaction problems over finite templates with few subpowers, which is based on a new type of consistency check,the affine consistency, which solves localized subproblems in $\operatorname{HS}(\mathbb{A})$, where $\mathbb{A}$ is the parametrizing algebra of the CSP.

The key ingredient of the algorithm was the notion of a $(B,\theta_B)$-test instance, where $B \in S(\mathbb{A})$ and $\theta_B\in \operatorname{Con}(B)$. In order to be able to define the test instance, (2,3)-consistency was used in order to establish the independence of the congruences on the domains of the test instance from the choice of the path pattern used. On the other hand, a weaker notion of local consistency was sufficient to solve the problem in the subinstances omitting affine behaviour; namely, the CSPs with congruence distributive templates can be defined in Linear  Datalog. 

\begin{openproblem} If $\mathcal{V}(\mathbb{A})$ is congruence modular, are there weaker local consistency notions which would suffice to define the congruences in the test instances for affine consistency, independently of the path patterns chosen?
\end{openproblem}

If this is indeed the case, we may be able to place the constraint satisfaction problems with templates with few subpowers in a complexity class which omits \textsc{P}-complete problems, e.g $\textsc{NC}^2$.

One of the questions which have motivated this line of investigation was the question which extensions of first-order logic, which are thought of as potential candidates for capturing polynomial time on finite relational structures, have the capability of expressing the solvability of CSPs with Maltsev templates. One such candidate logic is the Choiceless Polynomial Time with Counting (for more details, see \cite{gradel2015polynomial} ) In that article, the following problem was posed:

\begin{openproblem} Are all CSPs with templates with few subpowers expressible in the Choiceless Polynomial Time with Counting?
\end{openproblem}

\bibliography{digraph_reduction} 
\bibliographystyle{siam}

\end{document}